\documentclass{article}

\usepackage{arxiv}
\usepackage{cite}
\usepackage{amsmath,amssymb,amsfonts,bm,bbold,amsthm}
\usepackage{mathtools}
\usepackage{algorithmic}
\usepackage{textcomp}
\usepackage{dsfont}
\usepackage{tensor}
\usepackage{epstopdf}
\usepackage{orcidlink}
\usepackage{enumerate}
\usepackage{lscape}
\usepackage{multicol}
\usepackage{multirow}
\usepackage{booktabs}
\usepackage{framed}  
\usepackage{empheq}
\usepackage{graphics} 
\usepackage{epsfig} 
\usepackage{times} 
\usepackage{graphicx}
\usepackage{comment} 
\usepackage{apptools}
\usepackage{theoremref}
\usepackage{balance}
\graphicspath{{./},{./figures/}}

\newcommand{\fn}{{\mathrm{fn}}}
\newcommand{\init}{{\mathrm{in}}}
\newcommand{\GL}{\text{GL}^+(n,\mathbb{R})}

\DeclareMathOperator*{\oprod}{\overleftarrow{\prod}}
\newtheorem{thm}{Theorem}
\newtheorem{remark}{Remark}
\newtheorem{prop}{Proposition}
\newtheorem{lem}{Lemma}

\newtheorem{exmp}{Example}

\begin{document}


\title{
Collective steering in finite time:\\
controllability on $\GL$
}

\author{Mahmoud Abdelgalil\orcidlink{https://orcid.org/0000-0003-1932-5115}  and Tryphon T. Georgiou\orcidlink{https://orcid.org/0000-0003-0012-5447}
\thanks{Mahmoud Abdelgalil is with Electrical and Computer Engineering, University of California, San Diego, La Jolla, CA, USA, mabdelgalil@ucsd.edu.}
\thanks{Tryphon T. Georgiou is with Mechanical and Aerospace Engineering, University of California, Irvine, Irvine, CA, USA, tryphon@uci.edu.}}
\maketitle

\begin{abstract}
We consider the problem of steering a collection of {\it n} particles that obey identical {\it n\hspace*{1pt}}-\hspace*{1pt}dimensional linear dynamics via a common state feedback law towards a rearrangement of their positions, cast as a controllability problem for a dynamical system evolving on the space of matrices with positive determinant. 
We show that such a task is always feasible and, moreover, that it can be achieved arbitrarily fast. We also show that an optimal feedback control policy to achieve a similar feat, may not exist. Furthermore, we show that there is no universal formula for a linear feedback control law to achieve a rearrangement, optimal or not, that is everywhere continuous with respect to the specifications. We conclude with partial results on the broader question of controllability of dynamics on orientation-preserving diffeomorphisms.
\end{abstract}
%
\begin{keywords}{}
Linear Feedback Control, Controllability, Bilinear Systems, Optimal Control, Geometric Control.
\end{keywords}
%
\section{Introduction}
An often unappreciated fact is that optimal control policies, to steer linear dynamics between specified values of the state vector, may not always be expressible in \emph{linear} feedback form.  
The most elementary such example is the problem to steer first-order dynamics, $\dot x_t=u_t$ with $u_\cdot ,\,x_\cdot$ taking values in $\mathbb{R}$, from an initial state $x_{\rm in}=-1$ to a final $x_{\rm fn}=1$ over the time window $[0,1]$, seeking to minimize the quadratic cost functional $\int_0^1 u_t^2\,dt$. It can be readily seen that the optimal control is $u^\star_t=2$ and that the optimal state trajectory is $x^\star_t=2t-1$. If we now wanted to express the control input in linear state-feedback form, $u^\star_t=k_t x_t$, we would be called to use the feedback gain $k_t=(t-1/2)^{-1}$ for $t\in[0,2]$, which is not integrable, as the path needs to traverse past the point of equilibrium at the origin.

The source of the problem may at first seem to be the topological obstruction of the state trajectory having to cross the origin, since we demand traversing between terminal states that are on opposites sides and the state-space is $\mathbb{R}$. However the problem is much deeper and persists in higher dimensions due to a global topological obstruction, as we will explain later on, due to the fact that $\GL$ is not simply connected\footnotemark[1].
One consequence is that the optimal control policy, to traverse between specified states in a specified amount of time, cannot be expressed in feedback form in general.
Another consequence is that a linear feedback control policy to effect a specified (non-trivial) rearrangement of $n$ initial conditions in $\mathbb R^n$ for linear dynamics, via a time-varying feedback gain that depends continuously on the rearrangement linear map, may not exist either.
The existence of control laws that avoid such a phenomenon have been sought in earlier studies that aimed at the control of the Liouville equation, and are also of importance in the present work as well.

Our interest in feedback regulation to specify the state of a dynamical system stems from applications where one seeks to steer a collection of agents towards a desired  configuration. The agents individually are assumed to obey identical equations of motion that are linear, namely,
\begin{align}\label{eq:linear-sys0}
    \dot{x}{_t}&= A x{_t} + B u{_t}, 
\end{align}
with $x{_t}\in\mathbb{R}^n$, $A\in\mathbb{R}^{n\times n}$, $B\in\mathbb{R}^{n\times m}$ full-column rank, and $u{_\cdot}$ being the control input.
It is desirable that a time-varying matrix $K_\cdot\in\mathbb{R}^{m\times n}$ of common feedback gains which can be broadcast to all, to implement a state-feedback control, is sufficient to ensure that the swarm successively repositions itself to the desired terminal arrangement. We term this endeavour  {\em collective steering}.

Collective steering can be seen as a complementing theme to that of {\em ensemble control} \cite{li2010ensemble}.
In ensemble control, the main paradigm is to steer a collection of agents, each obeying their own dynamics, utilizing a universal input that can similarly be broadcast to all. One may visualize the agents as parametrized by individual ``labels'' that specify their identity and dynamics. This framework is tantamount to a Lagrangian viewpoint. In contrast, collective steering concerns the steering of a collection of agents obeying identical dynamics with each implementing their own input. It is the feedback law specifying these individual inputs that is being broadcast. Thereby, the framework of collective steering can be seen as representing an Eulerian viewpoint in regulating the motion of the collective.\footnotetext[1]{Throughout, $\GL$ denotes the identity component of the general linear group, the multiplicative group of square invertible matrices.}

Another salient feature in ensemble control is that the underlying global system, that includes all agents, is often severely under-actuated since the same input is implemented by all. On the other hand, in collective steering each agent has complete authority to specify their input using local information, i.e., utilizing knowledge of their own state. The roots of this formalism can be traced to the influential works by Brockett \cite{brockett2007optimal}, Agrachev and Caponigro \cite{agrachev2009controllability}, and, more recently, Agrachev and Sarychev \cite{agrachev2022control}, where they study controllability on the group of diffeomorphisms. In our recent work \cite{abdelgalil2024sub}, we explore an analogous formalism in connection to the \emph{holonomy} of optimal mass transport --loosely speaking, holonomy refers to the model that allows keeping track and regulating internal degrees of freedom of the mass distribution of particles that are viewed as distinguishable.

As noted, in the present work we limit our investigation to collective steering with linear feedback.
Thus, the problem we consider amounts to designing a common feedback gain matrix $K_\cdot$ and a reference signal $v_t$, so that the control law
\begin{equation}\label{eq:u}
    u_t=K_t x_t +v_t
\end{equation}
steers the individual states $x_t^i$, for $i\in\{1,\ldots,N\}$, of a swarm of $N$ agents 
from an initial configuration
\[
X_{\rm in}=\begin{bmatrix}
    x_{\rm in}^1, &\ldots& x_{\rm in}^N
\end{bmatrix} \mbox{ at } t=0,
\]
to a desired final one,
\[
X_{\rm fn}=\begin{bmatrix}
    x_{\rm fn}^1, &\ldots& x_{\rm fn}^N
\end{bmatrix} \mbox{ at }t=t_{\rm fn}.
\]
By substituting \eqref{eq:u} into \eqref{eq:linear-sys0},
\begin{align}\label{eq:linear-sys2}
    \dot{x}{_t}&= (A+BK_t) x{_t}+ Bv_t, 
\end{align}
we see that the common reference signal $v_t$ may be used to independently and freely adjust the mean value, i.e., the center of mass of the swarm. On the other hand, regarding the relative positioning of the agents, the most we can achieve with linear feedback is to effect a transformation 
\[
X_{\rm in}\mapsto \Phi{_{t_{\rm fn}}} X_{\rm in},
\]
with $\Phi_t$ being the state transition matrix of \eqref{eq:linear-sys2} at time $t$. Thus, without loss of generality, we specialize to $v=0$ and focus our attention on the case where $N=n$. 
The problem then reduces to steering the bi-linear control system
\begin{align}\label{eq:right-invariant-system0}
    \dot{X}_t&= (A+B K_t) X_t,
\end{align}
 with $X_t\in\mathrm{R}^{n\times n}$, from $X_0=X{_{\rm in}}$ to $X{_{t_{\rm fn}}}=X{_{\rm fn}}$, where now the feedback gain matrix $K_{\cdot}:[0,t_{\rm fn}]\rightarrow\mathbb{R}^{m\times n}$ represents our control authority. 
 
 Evidently, if $\det(X{_\fn})\det(X{_\init })\leq 0$, then necessarily $\det(X_t)$ must vanish for some $t\in[0,t_{\rm fn}]$. 
At such time, the swarm finds itself in a low dimensional subspace, very much as when we steer a single dynamical system through the origin in the opening example. For the same reasons, this situation cannot happen if the agents are steered collectively, with a choice of (bounded) $K_\cdot$. Therefore, we further restrict our attention to terminal specifications that satisfy 
\[
\det(X{_\fn})\det(X{_\init })> 0.
\]
Steering \eqref{eq:right-invariant-system0} between such terminal specifications is equivalent to steering
\begin{align}\label{eq:right-invariant-system}
    \dot{\Phi}_t&= (A+B K_t) \Phi_t,
\end{align}
from $\Phi_0=I$ to $\Phi{_{t_\fn}} = X{_\fn} X{_\init ^{-1}}$, with $\Phi_t$ being the state transition matrix for the specified dynamics.

Thus, henceforth, we investigate the controllability properties of \eqref{eq:right-invariant-system}, which is a \emph{right-invariant control system} evolving on the identity component of the (real) general linear group--i.e., matrices with positive determinant. We are interested especially in the notion of {\em strong controllability} as adapted\footnote{It is worth mentioning that strong controllability, as adopted herein, is referred to as \emph{exact-time controllability} in \cite[p. 89]{jurdjevic1997geometric} and that the notion of strong controllability adopted in \cite[Definition 9-(a), p. 88]{jurdjevic1997geometric} is slightly different. The distinction, however, is not crucial for our purposes and is omitted.} from \cite[p. 89]{jurdjevic1997geometric}, namely,
    the bi-linear system \eqref{eq:right-invariant-system} is said to be {\em strongly controllable} if, for any $t_\fn>0$ and $\Phi_\fn\in\GL$, there exists an integrable control input $K_\cdot:[0,t_\fn]\rightarrow\mathbb{R}^{m\times n}$ that steers \eqref{eq:right-invariant-system} from $\Phi_0=I$ to $\Phi_{t_\fn}=\Phi_\fn$.
This notion is to be contrasted with the classical notion of controllability wherein the terminal time $t_\fn$ may not be arbitrarily chosen, see, e.g., \cite[Definition 9-(b), p. 88]{jurdjevic1997geometric}.

While the literature is abound with necessary and sufficient conditions for controllability of control systems on Lie groups, going back to the pioneering works \cite{jurdjevic1972control,brockett1972system},  results on strong controllability are relatively sparse. The reason simply is that establishing strong controllability necessitates the complete characterization of the reachable set from the identity at each time instant $t>0$, also known as the exact-time reachable set. For \emph{driftless systems}, e.g., the case $A=0$ for \eqref{eq:right-invariant-system}, this is straightforward provided that the control input is allowed to take arbitrarily large values \cite[Theorem 5.1]{jurdjevic1972control}. For systems with drift, the situation is more involved and there are counterexamples that advise caution \cite[Example 11, p. 88]{jurdjevic1997geometric}. Although, under certain technical assumptions, it is possible to give an explicit characterization of the exact-time reachable set for systems with drift as in \cite{hirschorn1973topological}, such assumptions are far from general. Indeed, the bi-linear system \eqref{eq:right-invariant-system} does not satisfy the assumptions in \cite{hirschorn1973topological}. Further complications arise, even in the study of the weaker notion of controllability, due to the fact that $\GL$ is neither compact nor semi-simple as a Lie group. Nevertheless, as we prove below, it turns out that the Kalman rank condition is in fact a necessary and sufficient condition for strong controllability of \eqref{eq:right-invariant-system}.

In the body of the paper,
 we first establish the controllability of \eqref{eq:right-invariant-system} in Section \ref{sec:II}. In the process, we explain that classical optimal control falls short of providing a {\em direct} path to the controllability of \eqref{eq:right-invariant-system}, contrasting with the open-loop optimal control of an associated linear system. We return to this theme in Section \ref{sec:topology} where we discuss topological obstructions in constructing a universal formula for control laws that are continuous in the problem data.
In Section \ref{eq:strong} we refine the construction in Section \ref{sec:II} to establish strong controllability of \eqref{eq:right-invariant-system}. An illustrative example is presented in Section \ref{sec:example}, and the paper concludes with results on the controllability of dynamics that evolve on the space of orientation-preserving diffeomorphisms on $\mathbb R^n$.

\section{Collective Controllability}\label{sec:II}
A natural first step in addressing controllability properties of the matrix-valued dynamics in \eqref{eq:right-invariant-system}, that are bi-linear in nature,
is to view $U_t=K_t\Phi_t$ as a matrix-valued control input, and consider the linear dynamics $\dot \Phi_t=A\Phi_t + BU_t$ instead.
We do exactly that, but also introduce an additional constant parameter $K_c$, and write
\[
U_t = (K_t-K_c)\Phi_t.
\]
In this way, we examine the linear dynamical system
\begin{align}\label{eq:linear-sys}
    \dot{\Phi}{_t}&= A{_c} \Phi{_t} + B U{_t}, \mbox{ with } A{_c}=A+B K{_c},
\end{align}
having first chosen $K{_c}\in\mathbb{R}^{m\times n}$ suitably to ensure desirable properties for the spectrum of $A_c$ for the purposes of analysis. 

As can be readily seen, \eqref{eq:linear-sys} is obtained from \eqref{eq:right-invariant-system} by selecting the feedback gain
\begin{align}\label{eq:feedback-transformation}
    K{_t}&= U{_t} \Phi{_t^{-1}} + K_c.
\end{align}
Evidently, the transformation \eqref{eq:feedback-transformation} remains a bijection between $K_t$ and $U_t$ as long as $\Phi{_t}$ is invertible, i.e., as long as $\Phi{_t}\in\GL$. In this case, when the input $U\indices{_t}$ can be written in state-feedback form, the two systems \eqref{eq:right-invariant-system} and \eqref{eq:linear-sys}  follow identical trajectories when initialized at the same point.

Trajectories that coalesce under a control policy $U_\cdot$ or, are simply brought into a lower dimensional subspace at any time $t$, 
equivalent to $\Phi_t$ becoming singular, do not correspond to trajectories of \eqref{eq:right-invariant-system}. Hence, herein, we investigate control laws of \eqref{eq:linear-sys} that can be written in feedback form.

\subsection{Optimal control in feedback form?}

It is tempting to seek control laws for \eqref{eq:right-invariant-system} by adapting optimal control laws derived for \eqref{eq:linear-sys} for a quadratic performance criterion, since the latter can be expressed essentially in closed form. Unfortunately, this cannot be done. An optimal control policy may not be expressible in feedback form, in general.

The particular quadratic optimization criterion used to devise control laws is not essential, as similar issues arise for any alternative choice. Deeper reasons for that will become apparent later on, in Section \ref{sec:topology}, when we discuss topological obstructions to the existence of feedback laws that depent continuously on the problem data. At present and for simplicity, we restrict ourselves to control laws that minimize the control energy functional
\begin{align}\label{eq:intU2}
    U{_\cdot}\mapsto \int_0^{t_\fn}\lVert U{_t}\lVert_{\mathrm{F}}^2 \,\mathrm{d}t,
\end{align}
where $\lVert \cdot\lVert_{\mathrm{F}}$ denotes the Frobenius norm. We conclude this subsection with an example that highlights this point, that the optimal law may not be expressible in feedback form.

It is well known that as long as the pair $(A,B)$ satisfies the Kalman rank condition
\begin{align*}
    \mathrm{rank}(\left[B,~AB,~\cdots,~A^{n-1}B\right]) = n,
\end{align*}
which is equivalent to $(A_c,B)$ satisfying the rank condition,
the linear system \eqref{eq:linear-sys} is strongly controllable, in that it can be steered between any initial and final points $\Phi{_\init },\Phi{_\fn}\in\mathbb{R}^{n\times n}$ through an appropriate choice of the input $U{_\cdot}$ over an arbitrary interval $[0,t_\fn]$ with $t_\fn > 0$. The Kalman condition will be assumed throughout.

It is also well known that the control input $U{_\cdot^\star}$ that steers \eqref{eq:linear-sys} from $\Phi\indices{_\init }=I$ to $\Phi{_\fn}\in\GL$, over the interval $[0,t_\fn]$, and minimizes the cost functional \eqref{eq:intU2} is unique and given by the formula
\begin{align}\label{eq:minimum_energy_input}
    U{_t^\star}&= B{^\top}\mathrm{e}^{-A{_c^\top}t }W{_{t_\fn}^{-1}}\left(\mathrm{e}^{-A{_c} t_\fn } \Phi{_\fn} - I\right).
\end{align}
From the variation-of-constants formula, the corresponding optimal trajectory $\Phi{_\cdot^\star}$ is
\begin{align}\label{eq:minimum_energy_state}
    \Phi{_t^\star}= \mathrm{e}^{A{_c} t}\left(I +  W{_t} W{_{t_\fn}^{-1}}\left(\mathrm{e}^{-A{_c} t_\fn } \Phi{_\fn} - I\right)\right),
\end{align}
where the matrix-valued function (Grammian) $W_{\cdot}$ is
\begin{align}\label{eq:W}
    W_{\cdot}:t\mapsto \int_0^t\mathrm{e}^{-A{_c}\tau} B B{^\top}\mathrm{e}^{-A{_c^\top}\tau}\,\mathrm{d}\tau.
\end{align}

If $\det(\Phi{_t^\star})>0$ for all $t\in[0,t_\fn]$, i.e., $\Phi{_t^\star}\in\GL$, then the input
\begin{align}\label{eq:minimum-energy-feedback-gain}
K\indices{_\cdot^\star}:t\mapsto U{_t^\star} \left(\Phi{_t^\star}\right){^{-1}} + K_c,
\end{align}
steers the bi-linear system \eqref{eq:right-invariant-system} from $\Phi\indices{_0}=I$ to $\Phi{_\fn}\in\GL$, over the interval $[0,t_\fn]$. Such a conclusion, however, depends critically on $\Phi{_t^\star}$ remaining in $\GL$ for all $t\in[0,t_\fn]$. 
An elementary counterexample, along similar lines as the opening example of the introduction, where this fails to be the case is the following. In this, both $\Phi_\init,\Phi_\fn$ belong to the identity component of the general linear group, in that both have positive determinants.

\begin{exmp}\label{exmp:counter-example}
    Let $n\geq 2$ be even, and
    \begin{align*}
        A&=0,& B&=I, & K{_c}&=0, & \Phi{_\init }&=I, & \Phi{_\fn}&=-I.
    \end{align*}
    Because $n$ is even, we have that $\det(\Phi{_\init})=\det(\Phi{_\fn})=1$. Direct computation shows that $\Phi{_\cdot^\star}$ is given by
    \begin{align*}
        \Phi{_t^\star}&= \frac{t_\fn-2t}{t_\fn} I,
    \end{align*}
    which becomes singular at $t=\frac{1}{2}t_\fn$. The case of odd $n$ is not substantially different.
\end{exmp}\vspace*{0.1cm}

In light of this example, it is of great interest to obtain sufficient conditions on $\Phi{_\fn}$ for the optimal control policy $U{_t^\star}$ to be expressible in feedback form.

\subsection{Brockett's approach}\label{subsec:Brockett}

 In his seminal paper on the control of the Liouville equation \cite{brockett2007optimal}, Brockett suggested that the condition\footnote{Throughout,
 $\|\cdot\|$ denotes the induced $2$-norm (largest singular value).}
\begin{align}\label{eq:brockett-sufficient-condition}
    \| \Phi_\fn - I\| <1,
\end{align}
ensures that $\Phi_{t}$ as defined in
\eqref{eq:minimum_energy_state} remains invertible for all $t\in[0,t_\fn]$ and $t_\fn>0$. 
This is not true in general.
If it were true, it would guarantee that the feedback gain $K_\cdot$ in \eqref{eq:minimum-energy-feedback-gain} is well-defined, and would provide a convenient formula for designing piece-wise smooth trajectories in $\GL$ for the system \eqref{eq:right-invariant-system}. In this way, Brockett sought to establish controllability of \eqref{eq:right-invariant-system} using short optimal segments as motion primitives.
Below, we modify \eqref{eq:brockett-sufficient-condition} so as to complete Brockett's original program in Section \ref{subsec:next}.

The crux of Brockett's argument was to observe that, for any given $t_\fn>0$, there exists some constant gain $K{_c}\in\mathbb{R}^{m\times n}$ such that 
\begin{align}\label{eq:pole-placement-periodic}
    \mathrm{e}^{A_c t_\fn} &= I,
\end{align}
where, as before, $A_c= A+B K_c$. The existence of such a $K_c$ is guaranteed by the assumption that the pair $(A,B)$ satisfies the Kalman rank condition. With such a choice of $K{_c}$, the expression for the optimal trajectory \eqref{eq:minimum_energy_state} becomes
\begin{align}\label{eq:minimum_energy_state_prime}
    \Phi{_t^\star}= \mathrm{e}^{A{_c} t}\left(I +  W{_t} W{_{t_\fn}^{-1}}\left(\Phi_\fn - I\right)\right),\tag{\ref{eq:minimum_energy_state}$^\prime$}
\end{align}
Brockett then asserted \cite[Proof of Lemma 1]{brockett2007optimal} that the monotonicity
of the Grammians,
\begin{align}\label{eq:grammian-is-monotonic}
    0\prec W{_t} \prec W{_{t_\fn}}, \mbox{ for }t\in(0,t_\fn),
\end{align}
 implies that $\lVert W_t W_{t_\fn}^{-1}\lVert_{\vphantom{\mathrm{F}}} \leq 1$, and thereby, that \eqref{eq:brockett-sufficient-condition} suffices for
 \[
 \|W_tW_{t_\fn}^{-1}\left(\Phi_\fn - I\right)\|<1,
 \]
which in turn implies that $\det(\Phi{_t^\star})>0$ for all $t\in[0,t_\fn]$ and every $t_\fn>0$. 
However, in general, $W{_t} W{_{t_\fn}^{-1}}$ is not symmetric, and therefore the assertion that $\lVert W_t W_{t_\fn}^{-1}\lVert_{\vphantom{\mathrm{F}}} \leq 1$ is not true in general. The monotonicity in \eqref{eq:grammian-is-monotonic} only implies
that the spectral radius of $W{_t} W{_{t_\fn}^{-1}}$ is bounded above by $1$, and not the norm.

In the following proposition we strengthen Brockett's  condition so as to ensure that $\Phi{_t^\star}$ remains in $\GL$, and complete in the next section Brockett's program.

\vspace*{0.1cm}\begin{prop}
\label{prop:brockett-condition-corrected}
    Assume that the pair $(A,B)$ satisfies the Kalman rank condition, and that for a given $t_\fn>0$ a matrix $K_c$ has been selected so that \eqref{eq:pole-placement-periodic} holds.
    If $\Phi{_\fn}\in\GL$ is specified so that
 \begin{align}\label{eq:brockett-sufficient-condition-correct}
        \lVert W{_{t_\fn}^{-\frac{1}{2}}}\Phi{_\fn} W{_{t_\fn}^{\frac{1}{2}}}- I\lVert_{\vphantom{\mathrm{F}}} < 1,
    \end{align}
     then $\Phi_t^\star$ belongs to $\GL$ for all $t\in[0,t_\fn]$. 
    
\end{prop}\vspace*{0.1cm}
\begin{proof}
 The monotonicity property in \eqref{eq:grammian-is-monotonic} implies that
\begin{align*}
    0 \prec 
    W_{t_\fn}^{-\frac{1}{2}}
    W_t
    W_{t_\fn}^{-\frac{1}{2}}
    \prec I,
\end{align*}
and since
    $W_{t_\fn}^{-\frac{1}{2}}
    W_t
    W_{t_\fn}^{-\frac{1}{2}}$
    is symmetric, 
\begin{equation}\label{eq:Winequality}
    \lVert W_{t_\fn}^{-\frac12}
    W_t
    W_{t_\fn}^{-\frac12}
    \lVert < 1.
\end{equation}
From \eqref{eq:minimum_energy_state_prime},
\begin{align*}
\Phi_t^\star&=
    \mathrm{e}^{A{_c} t}
        \left(I +  W_t W_{t_\fn}^{-1}\left(\Phi_\fn -I\right)\right)\\
&=   \mathrm{e}^{A{_c} t}
        \left(I +  W_t W_{t_\fn}^{-\frac12}\left(W_{t_\fn}^{-\frac12}\Phi_\fn W_{t_\fn}^{\frac12} -I\right)W_{t_\fn}^{-\frac12}\right)\\
        &=   \mathrm{e}^{A{_c} t}W_{t_\fn}^{\frac12}
        \left(I + \phantom{W_{t_\fn}^{-\frac12}}\right.\\
        &\phantom{=xxx} \left. + W_{t_\fn}^{-\frac12} W_t W_{t_\fn}^{-\frac12}\left(W_{t_\fn}^{-\frac12}\Phi_\fn W_{t_\fn}^{\frac12} -I\right)\right)W_{t_\fn}^{-\frac12}.
\end{align*}
  From \eqref{eq:brockett-sufficient-condition-correct} and \eqref{eq:Winequality} it follows that
  \begin{equation}\label{eq:invertibility}
  \left(I + W_{t_\fn}^{-\frac12} W_t W_{t_\fn}^{-\frac12}\left(W_{t_\fn}^{-\frac12}\Phi_\fn W_{t_\fn}^{\frac12} -I\right)\right)
  \end{equation}
  is invertible and so is $\Phi_t^\star$.
\end{proof}

\begin{remark}
    The value of condition \eqref{eq:brockett-sufficient-condition-correct} is that it ensures that the optimal control law can be written in feedback form  \eqref{eq:minimum-energy-feedback-gain}. Thereby, it characterizes a class of matrices in $\GL$ for which this is possible and which can be reached from the identity via this particular choice of time-varying matrix of feedback gains.
    In the next section, matrices that belong to this class will be used as intermediate points to construct a path from the identity to any specified final value $\Phi_\fn$ in a similar manner. $\Box$
\end{remark}

\subsection{Controllability on $\GL$}\label{subsec:next}

We are now in position to complete Brockett's program and establish controllability of \eqref{eq:right-invariant-system}, i.e., that given an arbitrary $\Phi_\fn\in\GL$, there exists a suitable $K_\cdot$ that represents the control parameter and steers \eqref{eq:right-invariant-system} from the identity to $\Phi_\fn$.
This is stated next. The time $t_\fn$ required for the transition is not specified in advance. Strong controllability, where $t_\fn$ can in fact be specified in advance, will be established in Section \ref{eq:strong} via a refinement of the approach.

\begin{thm}\label{prop:brockett-controllability-corrected}
    Assume that the pair $(A,B)$ satisfies the Kalman rank condition. Then the bi-linear system \eqref{eq:right-invariant-system} is controllable on $\GL$.
\end{thm}

\begin{proof}
Let $\Phi_\fn$ be any matrix in $\GL$. We will show that there exists a time $t_\fn>0$ and a solution $\Phi_\cdot$ of \eqref{eq:right-invariant-system} for a suitable choice of control $K_\cdot$ so that $\Phi_{t_\fn}=\Phi_\fn$.

We arbitrarily select a time step $t_s>0$, and then a feedback $K_c$ such that 
\begin{align}\label{eq:pole-placement-periodictag}\tag{\ref{eq:pole-placement-periodic}$^\prime$}
    \mathrm{e}^{A_c t_s} &= I.
\end{align}
Next we compute $W_{t_s}$, and determine an integer $N$ together with a factorization
\[
\Phi_\fn = \oprod_{k=1:N} \Phi_k := \Phi_N\cdots \Phi_1,
\]
such that for all $k\in\{1,\ldots N\}$,
\[
\|\Phi_k - I \|< \sqrt{\frac{\lambda_{\rm min}(W_{t_s})}{\lambda_{\rm max}(W_{t_s})}}.
\]
That such a factorization is always possible for sufficiently large $N$ is a consequence of Lemma \ref{lem:phi-decomposition} that is provided in the Appendix. 
From Lemma \ref{lem:lambdas} that is also provided in the Appendix, setting 
\[
\Delta=\sqrt{\frac{\lambda_{\rm max}(W_{t_s})}{\lambda_{\rm min}(W_{t_s})}}(\Phi_k-I),
\]
we deduce that
\[
\|W_{t_s}^{-\frac12}\Phi_k W_{t_s}^{\frac12}-I\|<1,
\]
for all $k$. The statement in Proposition \ref{prop:brockett-condition-corrected}, applied to construct a path $\Phi^{k,\star}_t$, for $t\in[(k-1)t_s,kt_s]$, connecting $I$ to $\Phi_k$, allows us to conclude that there exists a trajectory 
\[
\Phi_t =
\begin{cases}
\Phi^{1,\star}_t, & \mbox{for }0\leq t \leq t_s,\\
\Phi^{2,\star}_t\Phi_1, & \mbox{for }t_s\leq t \leq 2t_s,\\
\hspace*{.5cm}\vdots\\
\Phi^{N,\star}_t \oprod_{k=2:N} \Phi_{k-1}, & \mbox{for }(N-1)t_s\leq t \leq Nt_s,
\end{cases}
\]
of the bi-linear system \eqref{eq:right-invariant-system} that connects the identity to $\Phi_\fn$ over the window $[0,N t_s]$, for terminal time $t_\fn=N t_s$.
\end{proof}

\begin{remark}
    The above proof follows the basic line sketched by Brockett \cite{brockett2007optimal}. However, Brockett assumed that the condition \eqref{eq:brockett-sufficient-condition} was sufficient to ensure that a feedback gain can be obtained to steer \eqref{eq:right-invariant-system} through a sequence of intermediate points 
    that terminate at the desired terminal condition. 
    Moreover, since the condition Brockett proposed was independent of $W_{t_s}$,
this led him to conclude that \eqref{eq:right-invariant-system} is strongly controllable. As we noted in Section \ref{subsec:Brockett}, since \eqref{eq:brockett-sufficient-condition} is not sufficient for control laws to be implementable in feedback form, neither conclusion follows. $\Box$
\end{remark}

\section{Strong Collective Controllability}\label{eq:strong}

As explained earlier, the ability to steer a collective between arbitrary configurations amounts to the controllability of the bilinear system \eqref{eq:right-invariant-system}. Accomplishing such a task arbitrarily fast on the other hand, amounts to strong controllability of \eqref{eq:right-invariant-system}. As we will show in the present section, interestingly, this stronger notion of controllability is also valid when the system matrices satisfy the Kalman rank condition.

In preparation to proving our claim on the strong controllability of \eqref{eq:right-invariant-system}, we provide first a refinement of Proposition \ref{prop:brockett-condition-corrected} that now allows us to identify a class of matrices in $\GL$ that we can reach  from the identity via suitable feedback gains  {\em arbitrarily fast}, that is, within any specified time $t_s>0$. These matrices will subsequently be used as motion primitives in a similar manner as before.

\begin{prop}\label{prop:georgiou-sufficient-condition}
    Assume that the pair $(A,B)$ satisfies the Kalman rank condition, $t_s>0$, and $K_c$ computed so that \eqref{eq:pole-placement-periodictag} holds. Then, for every $\Phi_\fn\in\GL$ that satisfies
\begin{subequations}\label{eq:georgiou-sufficient-condition}
        \begin{align}\label{eq:g1}
            W{_{t_s}^{-\frac{1}{2}}}\Phi{_\fn} W{_{t_s}^{\frac{1}{2}}} &= \left(W{_{t_s}^{-\frac{1}{2}}}\Phi{_\fn} W{_{t_s}^{\frac{1}{2}}}\right){^\top}, \\
            \label{eq:g2}
            W{_{t_s}^{-\frac{1}{2}}}\Phi{_\fn} W{_{t_s}^{\frac{1}{2}}}&\succ 0,
        \end{align}
    \end{subequations}
   it holds that $\Phi{_t^\star}\in \GL$, for  $t\in[0,t_s]$, while $\Phi_{t_s}=\Phi_\fn$.
\end{prop}
\begin{proof}
 We first claim that      
     \begin{equation}\label{eq:claim}
  I + \underbrace{W_{t_\fn}^{-\frac12} W_t W_{t_\fn}^{-\frac12}}_{S_t}\big(\underbrace{W_{t_\fn}^{-\frac12}\Phi_\fn W_{t_\fn}^{\frac12}}_{M} -I\big)
 \end{equation}
    is invertible under \eqref{eq:georgiou-sufficient-condition}. 
    To see this, note that the statement is trivially true when $t=0$, since $S_0=0$. 
    By continuity, the statement is true for some interval $[0,t_1)$. 
    The spectrum of $I+S_t(M-I)$ coincides with the spectrum of the matrix $I+S_t^\frac12 (M-I)S_t^\frac12$, by similarity transformation. From the fact that $0\prec S_t \preceq I$ for all $t\in(0,t_\fn]$, we have that
    \[
    I+S_t^\frac12 (M-I)S_t^\frac12 \succeq
    S_t^\frac12 M S_t^\frac12\succ 0.
    \]
    Having established that the expression in \eqref{eq:claim} is invertible, and inspecting 
     \eqref{eq:invertibility} in the proof of Proposition \ref{prop:brockett-condition-corrected}, 
     we observe that $\Phi^\star_t$ is invertible as well. This completes the proof.
\end{proof}

\begin{remark}
    It is interesting to note that the class of matrices that satisfy \eqref{eq:georgiou-sufficient-condition} is a rather thin set, since \eqref{eq:g1} is an algebraic constraint on its elements. Thus, \eqref{eq:georgiou-sufficient-condition} is more stringent than \eqref{eq:brockett-sufficient-condition-correct}. Yet, this class of functions is still unbounded. The structure of this class, in the symmetry condition \eqref{eq:g1} and the fact that it is unbounded, suffices to prove strong controllability by providing points for motion primitives as before. $\Box$
\end{remark}



Next, we bring attention to a relatively unknown fact, that any matrix with a positive determinant can be written as product of {\em finitely many} symmetric positive definite matrices. This fact follows from \cite[Proposition 4]{abdelgalil2024sub} that was established using Geometric Control. A rather remarkable result on such a factorization appeared much earlier in \cite{ballantine1967products,ballantine1968aproducts,ballantine1968bproducts} and was based on purely algebraic techniques. This earlier work established that at most five factors are needed. For specificity, we refer to this exact characterization on the number of factors in the proof of the subsequent theorem (Theorem \ref{thm:controllability-on-GL}).

\begin{thm}[Theorem 5 in \cite{ballantine1968bproducts}]\label{fact:ballantine-SPD-product}
    Every $\Phi\in \GL$ can be written as a product
    \begin{align*}
        \Phi = \oprod_{k=1:5} \Phi_k,
    \end{align*}
    for some symmetric positive definite matrices $\{\Phi_1,\dots,\Phi_5\}$. 
\end{thm}

Before estabilishing strong controllability, we need one more lemma, suitably recasting the statement of Theorem \ref{fact:ballantine-SPD-product}.

\begin{prop}\label{prop:SPD-product}
    Assume that the pair $(A,B)$ satisfies the Kalman rank condition, $t_s>0$, and that for some $K_c$ (without the need to satisfy \eqref{eq:pole-placement-periodic}) we compute $W_{t_s}$ as before.
    Then, any $\Phi_\fn\in\GL$ can be written as the product
    \begin{align*}
        \Phi_\fn = \oprod_{k=1:5}\Phi_k,
    \end{align*}
    of matrices
    $\{\Phi_1,\Phi_2,\dots,\Phi_5\}$ that each satisfies \eqref{eq:georgiou-sufficient-condition}.
\end{prop}
\begin{proof}
Apply Theorem \ref{fact:ballantine-SPD-product} to conclude that
\[
W_{t_\fn}^{-\frac12} \Phi_\fn W_{t_\fn}^{\frac12}=\oprod_{k=1:5} \Psi_k
\]
for symmetric positive definite  $\Psi_k$, for $k\in\{1,\ldots,5\}$, and define $\Phi_k= W_{t_\fn}^{\frac12}\Psi_kW_{t_\fn}^{-\frac12}$.
\end{proof}

We are now in a position to prove that under the Kalman rank condition the bi-linear system \eqref{eq:right-invariant-system} is strongly controllable on $\GL$.

\vspace*{0.1cm}\begin{thm}\label{thm:controllability-on-GL}
    The following statements are equivalent:
    \begin{enumerate}
      \item[(i)] The pair $(A,B)$ satisfies the Kalman rank condition.     
      \item[(ii)] System \eqref{eq:right-invariant-system} is controllable on $\GL$.
        \item[(iii)] System \eqref{eq:right-invariant-system} is strongly controllable on $\GL$.
    \end{enumerate}
\end{thm}\vspace*{0.1cm}
\begin{proof}
The implications (iii) $\Rightarrow$ (ii) and (ii) $\Rightarrow$ (i) are trivial. Thus, we only need to prove that (i) $\Rightarrow$ (iii).

Consider an arbitrary $\Phi_\fn\in\GL$ and any $t_\fn>0$.
Let $K_c$ be chosen so that \eqref{eq:pole-placement-periodic} holds for $t_s=t_\fn/5$.
Compute $W_{t_s}$ as before, using \eqref{eq:W}.
From Proposition \ref{prop:SPD-product}, there exist five positive definite matrices $\{\Phi_1,\dots,\Phi_5\}$ that satisfy \eqref{eq:georgiou-sufficient-condition} for $t_\fn/5$, and are such that
    \begin{align*}
        \Phi_\fn = \oprod_{k=1:5} \Phi_k.
    \end{align*}
    From Proposition \ref{prop:georgiou-sufficient-condition}, there exists respective feedback gains $K{_\cdot^k}:[0,t_\fn/5]\rightarrow\mathbb{R}^{m\times n}$ that steer \eqref{eq:right-invariant-system} from $I$ to $\Phi_{k}$. By concatenating the application of the feedback gains $K{_\cdot^k}$, over successive intervals of duration $t_s$, we obtain a control gain $K{_\cdot}:[0,t_\fn]\rightarrow\mathbb{R}^{m\times n}$ that steers \eqref{eq:right-invariant-system} from $I$ to $\Phi_\fn$ over the specified interval $[0,t_\fn]$. Since $\Phi_\fn\in\GL$ and $t_\fn>0$ are arbitrary, it follows that \eqref{eq:right-invariant-system} is strongly controllable.
\end{proof}

\begin{remark}
An alternative path to proving strong controllability can be based on some deep  and technical results in geometric control  \cite{jurdjevic1997geometric}. We briefly sketch this alternative line of arguments. A necessary and sufficient condition \cite[Theorem 12, p.\ 89]{jurdjevic1997geometric} for strong controllability is that the so-called \emph{strong Lie saturate} (see \cite[Definition 8, p.87]{jurdjevic1997geometric}) of the family of vector fields corresponding to \eqref{eq:right-invariant-system} generates the entire tangent space of $\GL$ at $\Phi$, for every $\Phi\in\GL$. To verify that this is the case requires explicit construction of the strong Lie saturate. Carrying out the necessary steps is a lengthy and rather involved process, that we chose not to include herein--one needs to explicitly construct a sequence of \emph{prolongations} of the family of vector fields corresponding to \eqref{eq:right-invariant-system} in a manner that preserves the closure of reachable sets. The procedure is similar to the one Jurdjevic outlines in the proof of \cite[Theorem 10 in Chapter 3]{jurdjevic1997geometric}). $\Box$
\end{remark}

\section{A topological obstruction to the continuity of a universal feedback gain}\label{sec:topology}

As noted earlier, \eqref{eq:minimum_energy_input} provides a universal expression for an {\em open-loop control law} that steers the linear dynamics \eqref{eq:linear-sys} from the identity to $\Phi_\fn$. This expression is continuous in the problem specifications, namely, in $\Phi_\fn$. We also saw that a similar construction for a {\em feedback control law}, in the form of a time-varying feedback gain matrix $K_\cdot$, is not possible, at least when relying on open-loop optimal control. The constructions that were used to establish controllability and strong controllability, provide no guarantee that the law is continuous in the problem data. Thus, it is natural to ask whether such a closed-form continuous expression for some feedback gain matrix $K_\cdot$ is at all possible for \eqref{eq:right-invariant-system} and, more broadly, in what instances such an endeavour may admit an answer in the affirmative.

In the process of explaining why such a formula {\em does not exist} for  the bi-linear system \eqref{eq:right-invariant-system}, it is instructive to compare first with \eqref{eq:linear-sys}, and second, with the continuous-time Lyapunov equation
\begin{align}\label{eq:lyapunov-equation}
    \dot{\Sigma}{_t}&= (A+BK{_t})\Sigma{_t} + \Sigma{_t} (A+B K{_t})^\top.
\end{align}
In the case of \eqref{eq:lyapunov-equation} the control parameter is again $K_\cdot$, and the task is to steer\footnote{$\mathrm{Sym}^+(n)$ denotes the cone of symmetric positive definite matrices.} $\Sigma{_t}\in\mathrm{Sym}^+(n)$, from the identity to some $\Sigma_\fn$. For both of these cases, \eqref{eq:linear-sys} and \eqref{eq:lyapunov-equation}, control laws that are continuous with respect to the specifictions do exist.
Considering these three cases (\ref{eq:linear-sys},\ref{eq:lyapunov-equation},\ref{eq:right-invariant-system}) together helps underscore the nature of the topological obstruction for the case of \eqref{eq:right-invariant-system}.

\begin{remark}
  A goal in Brockett's work \cite{brockett2007optimal} was to ascertain first the controllability of \eqref{eq:right-invariant-system} in order to establish the controllability of 
\eqref{eq:lyapunov-equation}, which amounts to
the Liouville equation for the case of linear dynamics. Brockett's original plan carries through, in that the (strong) controllability of \eqref{eq:lyapunov-equation} can be deduced from the (strong) controllability of \eqref{eq:right-invariant-system}, that has already been established in Theorem \ref{thm:controllability-on-GL}. The pertinent argument is detailed in Lemma \ref{lem:lyapunov-equation-controllability} in the Appendix. We note in passing that, following an independent route, the strong controllability of \eqref{eq:lyapunov-equation} has been established in \cite{chen2015optimal,chen2016optimal}.
We also note that Brockett further conjectured \cite[Remark 1, p.\ 30]{brockett2007optimal} that there should be a solution for \eqref{eq:lyapunov-equation} and \eqref{eq:right-invariant-system} that depends smoothly on the problem data, and while the first is valid the second fails.  
$\Box$\end{remark}

\subsection{Continuity of control laws to specifications}

We begin by providing an abstract formulation for the problem at hand,
to assess whether a solution that is continuous in the problem data exists.

\begin{prop}
Consider a continuous system of controlled dynamics that evolves on a state manifold $\mathcal S$. Assume further that the system is strongly controllable and that a continuous control law can be specified as a function of any terminal states $S_\fn \in \mathcal S$  that can generate a state trajectory $S_\cdot$ from a given $S_\init$ to $S_\fn$, over an interval $[0,t_\fn]$. Then, $\mathcal S$ is contractible.
\end{prop}

\begin{proof}
These assumptions amount to the existence of a continuous map
\[
\mathfrak C\;:\; [0,t_\fn]\times \mathcal S \to \mathcal S
\]
with the following properties:
\begin{itemize}
    \item[i)] $\mathfrak{C}(0,S_\fn) = S_\init $, and $\mathfrak{C}(t_\fn,S_\fn) = S_\fn$
    \item[ii)] $\mathfrak{C}(t,S_\fn) \in \mathcal S$, for all $t\in[0,t_\fn]$.
\end{itemize}
The curves $\mathfrak{C}(\cdot,S_\fn)$ are continuous and represent trajectories specified by the control law as functions of $S_\fn$.
This map can readily provide a homotopy to continuously deform
any loop to the point $S_\init$. Equivalently, $\mathfrak C$ is a continuous homotopy between the constant map $S_\fn\to S_\init$ and the identity map $S_\fn\to S_\fn$.
Thus, $\mathcal S$ must be contractible.
\end{proof}

The assumption of strong controllability, as opposed to only controllability, allows the terminal time $t_\fn$ to be independent of $S_\init,S_\fn$. The claim, that $\mathcal S$ is contractible as a topological space amounts to the fundamental group $\pi_1(\mathcal S)$, which consists of equivalent classes of loops in the space under homotopy, is trivial, in that any loops can be continuously shrank to a point \cite[Chapter 3]{singer2015lecture} as sketched in the proof.

\subsection{Representative cases}
Next, we interpret the statement of the proposition for the three cases of interest, the dynamics in (\ref{eq:linear-sys},\ref{eq:lyapunov-equation},\ref{eq:right-invariant-system}).


\subsubsection*{Linear dynamics \eqref{eq:linear-sys}}
In the case of
 \begin{align}\label{eq:linear-sysprime}
     \dot{\Phi}{_t}= A{_c} \Phi{_t} + B U{_t},
     \tag{\ref{eq:linear-sys}$^\prime$}
\end{align}
with control parameter is $U_t\in\mathbb R^{m\times n}$ and state parameter $\Phi_t\in\mathbb R^{n\times n}$, the state space is 
$\mathcal S=\mathbb R^{n\times n}$ which is  contractible. It is readily seen that
\eqref{eq:minimum_energy_input}-\eqref{eq:minimum_energy_state} define a smooth homotopy $\mathfrak{C}:[0,t_\fn]\times \mathbb{R}^{n\times n}\rightarrow\mathbb{R}^{n\times n}$ for which
the conditions i-ii) hold.

\subsubsection*{Lyapunov dynamics \eqref{eq:lyapunov-equation}}
We observe that the dynamics in this case are not linear. Yet, the states $\Sigma_t\in \mathrm{Sym}^+(n)$, and the state space $\mathcal S=\mathrm{Sym}^+(n)$ is contractible.\footnote{To see that $\mathrm{Sym}^+(n)$ is contractible, observe that $\Sigma \mapsto (1-t)I+\Sigma$ is a continuous homotopy between the constant and identity maps, for $t\in[0,1]$.}
Thus, the existence of a control law that depends continuously on $\Sigma_\fn$  is not ruled out by the topology of $\mathcal S$. Indeed, a choice for such a control law is given in the proof of the following proposition.

\begin{prop}
    Consider the dynamical equation  \eqref{eq:lyapunov-equation}, and let $t_\fn>0$, and $\Sigma_\init,\Sigma_\fn\in \mathrm{Sym}^+(n)$.
There exists a control law $K^\star_\cdot$ that continuously depends on the problem data, such that the solution $\Sigma^\star_\cdot$ to 
 \eqref{eq:lyapunov-equation} satisfies 
\begin{itemize}
    \item[i)] $\Sigma^\star_0 = \Sigma_\init $, and $\Sigma^\star_{t_\fn}=\Sigma_\fn$
    \item[ii)] $\Sigma^\star_t\in \mathrm{Sym}^+(n)$, for all $t\in[0,t_\fn]$.
\end{itemize}
\end{prop}
\begin{proof}
    Let $K_c\in\mathbb{R}^{m\times n}$ be such that \eqref{eq:pole-placement-periodic} is satisfied and compute $W_{t_\fn}$, as before. Define
    \begin{align*}
        \tilde{\Sigma}{_\init }&:= W{_{t_\fn}^{-\frac{1}{2}}} \Sigma{_\init }W{_{t_\fn}^{-\frac{1}{2}}}, \;\; \tilde{\Sigma}{_\fn}:= W{_{t_\fn}^{-\frac{1}{2}}} \Sigma_\fn W{_{t_\fn}^{-\frac{1}{2}}},\mbox{ and}\\
        \Phi{_\fn}&:=W{_{t_\fn}^{\frac{1}{2}}}\tilde{\Sigma}{_\init ^{-\frac{1}{2}}}\left(\tilde{\Sigma}{_\init ^{\frac{1}{2}}}\tilde{\Sigma}{_\fn}\tilde{\Sigma}{_\init ^{\frac{1}{2}}}\right)^{\frac{1}{2}}\tilde{\Sigma}{_\init ^{-\frac{1}{2}}} W{_{t_\fn}^{-\frac{1}{2}}}.
    \end{align*}
    By construction, $\Phi{_\fn}$ satisfies \eqref{eq:georgiou-sufficient-condition} and therefore, by Proposition \ref{prop:georgiou-sufficient-condition}, the solution $\Phi^\star_\cdot$ to \eqref{eq:right-invariant-system} for 
    \[
    K^\star_t= B{^\top}\mathrm{e}^{-A{_c^\top}t }W{_{t_\fn}^{-1}}\left(\Phi{_\fn} - I\right)(\Phi_t^\star)^{-1} + K_c,
    \]
    remains in $\GL$ for all $t\in[0,t_\fn]$. It can be seen that 
    \begin{align*}
        \Phi{_0^\star}\Sigma{_\init }(\Phi{_0^\star})^\top &= \Sigma{_\init }, \mbox{ and }\Phi{_{t_\fn}^\star}\Sigma_\init(\Phi{_{t_\fn}^\star})^\top = \Sigma_\fn,
    \end{align*}
    and, by direct differentiation, that 
    \[
    \Sigma_t^\star:= \Phi_t^\star \Sigma_\init (\Phi_t^\star)^\top
    \]
    satisfies \eqref{eq:lyapunov-equation} and remains in $\mathrm{Sym}^+(n)$ over $[0,t_\fn]$. Finally, $K^\star_\cdot,\Sigma_\cdot^\star$ depend smoothly on the problem data and time.
\end{proof}

The construction in the above proof is based on a suitable modification of McCann's displacement interpolation \cite{mccann1997convexity} that was introduced in \cite{chen2016optimal} for cases when the control authority is constrained to channel through an input matrix $B$.

\subsubsection*{Right-invariant bilinear dynamics \eqref{eq:right-invariant-system}}

For the case of \eqref{eq:right-invariant-system} the state space is $\mathcal S=\GL$. However, by the polar decomposition, $\GL$ is homeomorphic to the Cartesian product of $\mathrm{Sym}^+(n)$ with the special orthogonal group $\mathrm{SO}(n)$. Since $\mathrm{Sym}^+(n)$ is contractible, $\pi_1(\GL)=\pi_1(\mathrm{SO}(n))$.
It is well known that $\mathrm{SO}(n)$ is not contractible,
and that in fact, $\pi_1(\mathrm{SO}(2))=\mathbb Z$ and
$\pi_1(\mathrm{SO}(n))=\mathbb Z_2$ for any $n\geq 3$.
Thus, $\GL$ is not contractible (except for $n=1$), and therefore a formula for a universal control law, as sought, does not exist.

\section{Illustrative example}\label{sec:example}
\begin{figure*}
    \centering
    \begin{minipage}[t]{0.485\textwidth}
        \centering
        \includegraphics[width=1\linewidth]{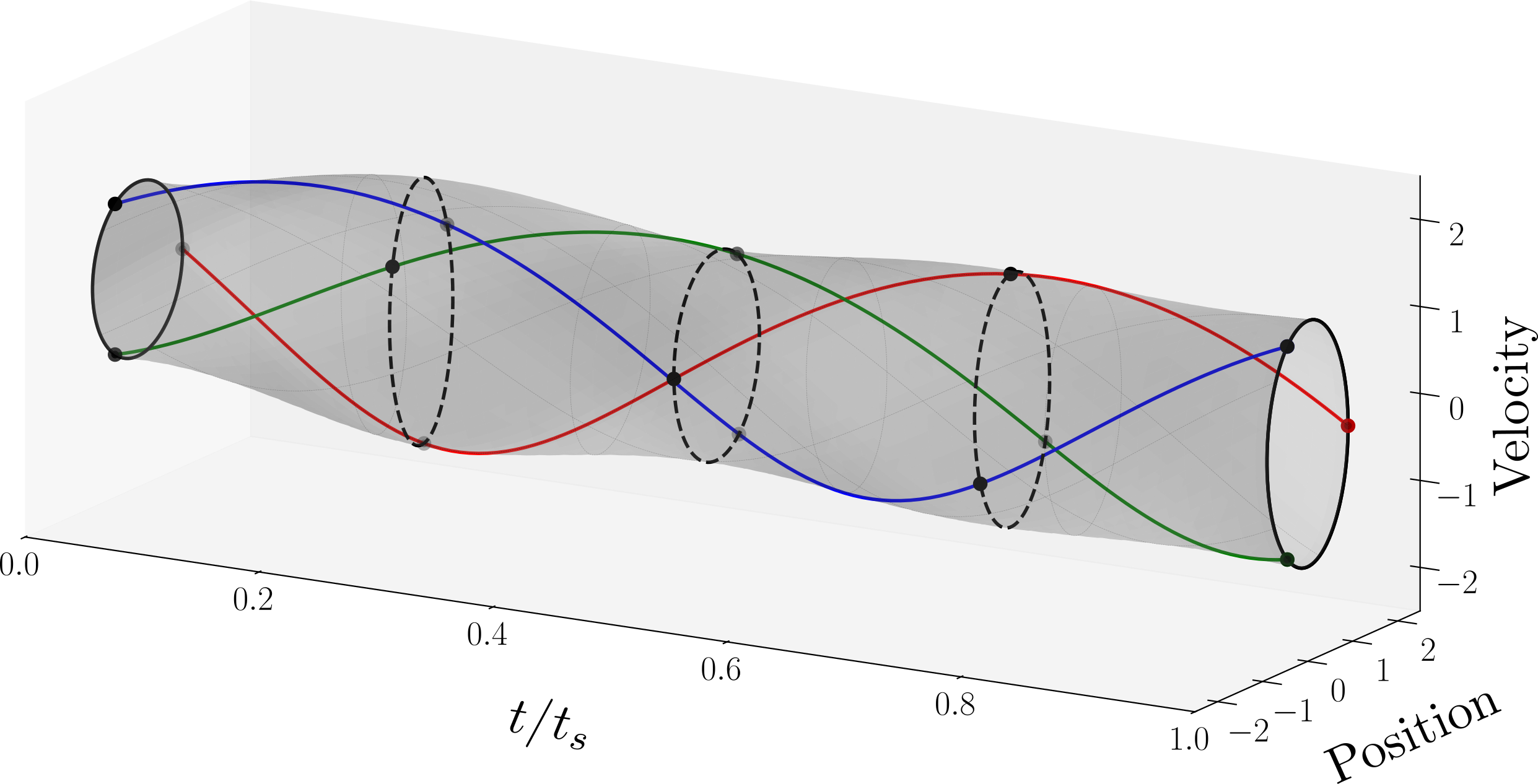}
        \caption{The first segment, i.e. $\Phi_{t}^{1,\star}$ with $t\in[0,t_s]$, in the piecewise smooth curve \eqref{eq:piecewise-smooth-trajectory-exmp}, which starts from $I$ and ends at $\Phi_1$.}
        \label{fig:first-segment-exmp}
    \end{minipage}
    \hfill
    \begin{minipage}[t]{0.485\textwidth}
        \centering
        \includegraphics[width=1\linewidth]{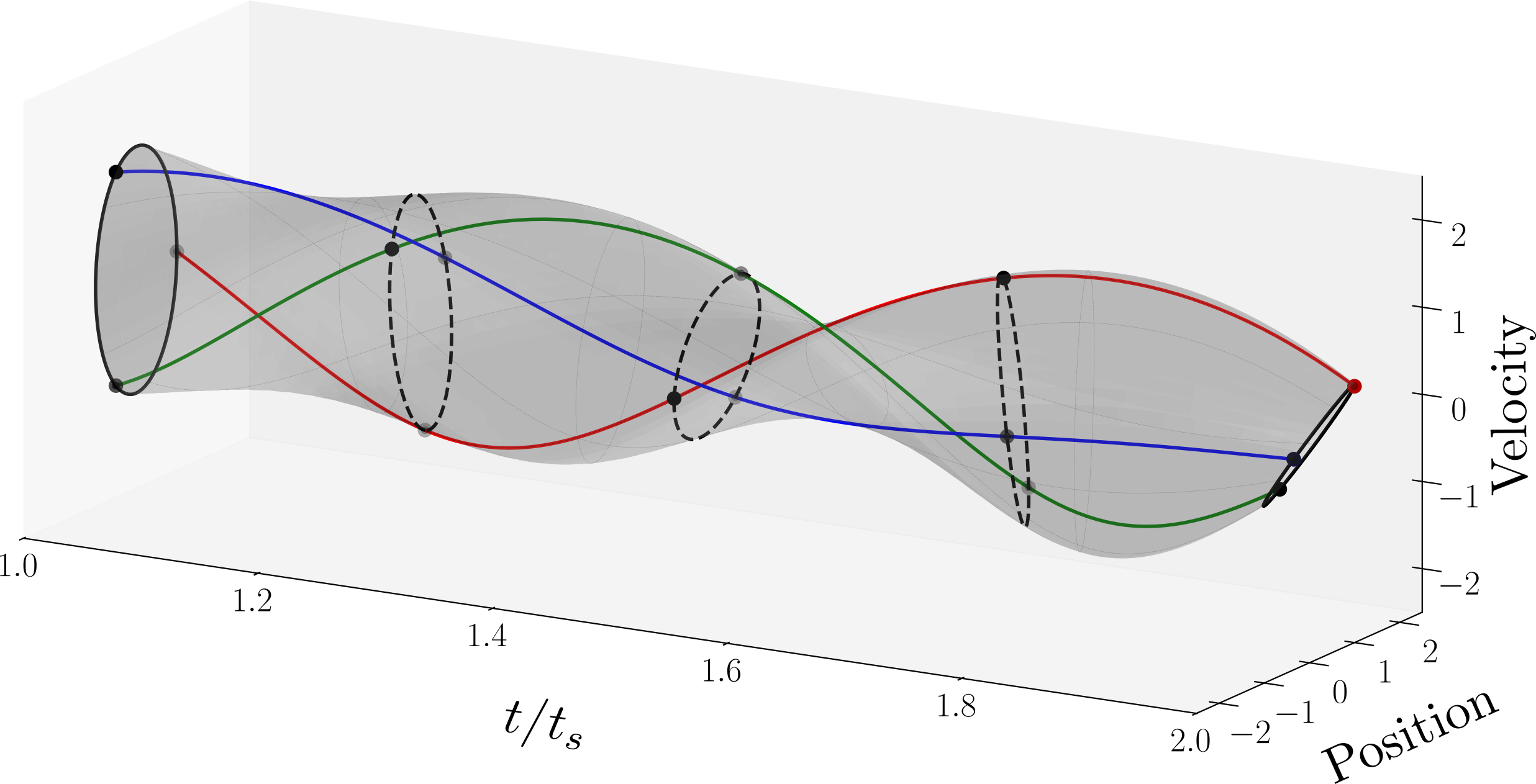} 
        \caption{The second segment, i.e. $\Phi_{t-t_s}^{1,\star}$ with $t\in[t_s,2t_s]$, in the piecewise smooth curve \eqref{eq:piecewise-smooth-trajectory-exmp}, which starts from $\Phi_1$ and ends at $\Phi_2$.}
        \label{fig:second-segment-exmp}
    \end{minipage}
\end{figure*}
In this section, we provide a numerical case study to illustrate the theoretical results presented thus far. Henceforth, we fix $n=2$ and
\begin{align*}
    A&= \begin{bmatrix}
        0 & 1\\ 0 & 0
    \end{bmatrix}, & B&= \begin{bmatrix}
        0 \\ 1\end{bmatrix},
\end{align*}
in \eqref{eq:linear-sys0}, which corresponds to a double integrator, i.e., an inertial system with force as input obeying Newton's law. Steering \eqref{eq:right-invariant-system} amounts to effecting a re-arrangement in phase space, i.e., of the positions and velocities of a collection of inertial particles. 

For any $t_s>0$, we see that the constant gain
\begin{align}\label{eq:K_c-example}
    K_c&=\begin{bmatrix}-\frac{4\pi^2}{t_s^{2}} & 0\end{bmatrix},
\end{align}
guarantees that \eqref{eq:pole-placement-periodictag} holds. Then, with \eqref{eq:K_c-example} as the choice of the constant gain $K_c$, explicit computation gives
\begin{align*}
    W_t W_{t_s}^{-1}&= \frac{t}{t_s} I - \frac{1}{2\pi}\begin{bmatrix}\frac{1}{2}\sin\left(\frac{4\pi t}{t_s}\right) & \frac{t_s}{2\pi}\sin^2\left(\frac{2\pi t}{t_s}\right)\\
    \frac{2\pi}{t_s}\sin^2\left(\frac{2\pi t}{t_s}\right)   & -\frac{1}{2}\sin\left(\frac{4\pi t}{t_s}\right)
    \end{bmatrix},
\end{align*}
for any $t_s>0$ and for all $t\in[0,t_s]$, which is not symmetric, in general. Indeed, at $t=t_s/4$, we see that
\begin{align*}
    W_{\frac{1}{4}t_s} W_{t_s}^{-1}= \frac{1}{4}\begin{bmatrix}
        1 & -\frac{t_s}{\pi^2}\\
        \frac{1}{t_s} & 1
    \end{bmatrix},
\end{align*}
and, as can be verified, we have that
\begin{align*}
    \lim_{t_s\searrow 0}\|W_{\frac{1}{4}t_s} W_{t_s}^{-1}\|= +\infty.
\end{align*}
We now seek a feedback gain $K_\cdot$ that steers \eqref{eq:right-invariant-system} between 
\begin{align*}
    \Phi_0&= I, \mbox{ and } \Phi_\fn= \mathrm{e}^{\frac{\pi}{4}\Omega},
\end{align*}
with $\Omega$ taken to be the skew-symmetric matrix
\begin{align*}
    \Omega=\begin{bmatrix}
        0 & -1\\ 1 & \hphantom{-}0
    \end{bmatrix}.
\end{align*}
In this case, $\Phi_\fn$ amounts to a $45^\circ$ planar rotation. To proceed, we define the symmetric positive definite matrices
\begin{align*}
    S_1&= \begin{bmatrix} 1 & \sigma\\ \sigma & 2\sigma^2
    \end{bmatrix},\\ S_2&= \begin{bmatrix} 2\sigma^2 & \sigma\\ \sigma & 1
    \end{bmatrix},\\
    S_3&= (S_2 S_1^2 S_2)^{-\frac12}S_2 S_1,
\end{align*}
taking $\sigma = \frac14 (\sqrt{17}-3)$. Direct computation shows that
\begin{align*}
    S_3S_2S_1=\Phi_\fn.
\end{align*}
and, therefore, we see that
\begin{align*}
    \Phi_\fn&= \oprod_{k=1:5} \Phi_k,
\end{align*}
where the matrices $\Phi_k$ for $k\in\{1,2,3,4,5\}$ are given by
\begin{align*}
    \Phi_1&=W_{t_s}^{\frac{1}{2}}, & \Phi_2&=W_{t_s}^{\frac{1}{2}}S_1W_{t_s}^{-\frac{1}{2}}, \\ \Phi_3&=W_{t_s}^{\frac{1}{2}}S_2W_{t_s}^{-\frac{1}{2}}, & \Phi_4&=W_{t_s}^{\frac{1}{2}}S_3W_{t_s}^{-\frac{1}{2}}, \\ 
    \Phi_5&= W_{t_s}^{-\frac{1}{2}},
\end{align*}
for every $t_s>0$. Each of these factors, by construction, satisfies the sufficient condition \eqref{eq:georgiou-sufficient-condition}. Therefore, the piecewise smooth curve
\begin{align}\label{eq:piecewise-smooth-trajectory-exmp}
    \Phi_t =
    \begin{cases}
    \Phi^{1,\star}_t, & \mbox{for }0\leq t \leq t_s,\\
    \Phi^{2,\star}_{t-t_s}\Phi_1, & \mbox{for }t_s\leq t \leq 2t_s,\\
    \hspace*{.5cm}\vdots\\
    \Phi^{5,\star}_{t-4t_s} \oprod_{k=2:5} \Phi_{k-1}, & \mbox{for }4t_s\leq t \leq 5t_s,
    \end{cases}
\end{align}
with each segment $\Phi^{k,\star}_\cdot:[0,t_s]\rightarrow \GL$ given by
\begin{align*}
    \Phi^{k,\star}_t:= \mathrm{e}^{A{_c} t}\left(I +  W{_t} W{_{t_s}^{-1}}\left(\Phi{_k}- I\right)\right),
\end{align*}
for $k\in\{1,2,3,4,5\}$,
is a trajectory of the bilinear system \eqref{eq:right-invariant-system}. The corresponding input $K_{\cdot}$ can be computed from \eqref{eq:right-invariant-system}. 

An illustration of the first two segments of the trajectory \eqref{eq:piecewise-smooth-trajectory-exmp} is shown in Figure \ref{fig:first-segment-exmp} and Figure \ref{fig:second-segment-exmp}, respectively. To generate the figures, we use $t_s=4$. We emphasize, however, that the construction provided above works for any $t_s>0$, although the corresponding trajectories may be difficult to visualize. In both figures, the colored lines represent the trajectories of three different ``tracer" particles as the collective undergoes the motion defined by the curve \eqref{eq:piecewise-smooth-trajectory-exmp}.

\section{On the controllability of orientation-preserving diffeomorphisms}\label{sec:diffeomorphisms}

Naturally, our interest in collective steering extends beyond
the linear group of diffeomorphisms $\GL$, to the general group of {\em orientation-preserving diffeomorphisms} $\mathrm{Diff}^+(\mathbb{R}^n)$ of $\mathbb R^n$ that may allow a more versatile repositioning of a collective to a specified terminal configuration. Thus, it is of interest to study the controllability of
\begin{equation}\label{eq:rightinvariant}
    \frac{\partial \varphi_t(x)}{\partial t} = (A\cdot+Bu_t(\cdot))\circ\varphi_t(x).
\end{equation}
In this, as well as in the earlier part of our study, we have been inspired by the work of Brockett \cite{brockett2007optimal} who drew attention to problems pertaining to the controllability of the Liouville equation. Other related works include \cite{agrachev2009controllability,agrachev2022control} and, more recently, \cite{tabuada,raginsky}.

In the present section, developing on the basic idea in Brockett \cite{brockett2007optimal} to devise control primitives that can be used to stitch together a path between states, we establish a reachability result for a class of Lipschitz diffeomorphisms via an approach analogous to that in Section \ref{subsec:next}.

Once again, we assume that the pair $(A,B)$ satisfies the Kalman rank condition and consider the linear system
\begin{align}\label{eq:linear-system-diff}
    \dot{x}&= A_c x + B u, \;\; A_c= A+B K_c,
\end{align}
with $x\in\mathbb R^n$, $ u\in\mathbb R^m$, and $K_c\in\mathbb{R}^{m\times n}$ chosen to satisfy \eqref{eq:pole-placement-periodictag} for some specified $t_s>0$.
The open-loop control input $u_\cdot$ that steers \eqref{eq:linear-system-diff} between terminal conditions $x_\init$ to $x_\fn$ over the interval $[0,t_s]$, selected to minimize
    $\int_0^{t_s} \|u_t\|^2\,\mathrm{d}t$,
is
\begin{align*}
    u^\star_t&= B^\top \mathrm{e}^{-A_c^\top t} W_{t_s}^{-1}(
    x_\fn - x_\init),
\end{align*}
with corresponding state trajectory
\begin{align*}
    x^\star_t&= \mathrm{e}^{A_c t}\left(x_0 + W_t W_{t_s}^{-1}(
    x_\fn - x_\init)\right),
\end{align*}
where, as before, $W_{\cdot}$ is the Grammian specified in \eqref{eq:W}.

If now $\varphi\in \mathrm{Diff}^+(\mathbb{R}^n)$,
then the optimal state trajectory and the associated optimal input connecting $x_\init$ to $\varphi(x_\init)$ over the interval $[0,t_s]$, for all $x_\init\in\mathbb{R}^n$, are given by
\begin{subequations}\label{eq:optimal-pair-diff}
    \begin{align}
        x^\star_t&= \mathrm{e}^{A_c t}\left(x_\init + W_t W_{t_s}^{-1}(
        \varphi(x_\init) - x_\init)\right),\label{eq:optimal-state-diff}\\
        u^\star_t&= B^\top \mathrm{e}^{-A_c^\top t_s} W_{t_s}^{-1}(
        \varphi(x_\init) - x_\init).\label{eq:optimal-input-diff}
    \end{align}
\end{subequations}
Brockett \cite{brockett2007optimal} proposed that if the map $x\mapsto \varphi(x)-x$ is a contraction,
then \eqref{eq:optimal-input-diff} can be written in feedback form
\begin{align}\label{eq:feedback-form-diff}
    u^\star_t = K(x^\star_t,t),
\end{align}
for a suitable continuous feedback law $K:\mathbb{R}^n\times[0,t_s]\rightarrow \mathbb{R}^m$,
for all $x_\init\in\mathbb{R}^n$ and all $t\in[0,t_s]$. The argument in \cite[Lemma 1]{brockett2007optimal} relied on the assumption that
    $\|W_t W_{t_s}^{-1}\|\leq 1$,
which is not true in general. However, Brockett's line of development
can be carried out with a modified condition on the map $\varphi$, similar to Proposition \ref{prop:brockett-condition-corrected}, as stated next.

\begin{prop}
    Under the standing assumptions on $(A,B)$ and $t_s$,
    if $\varphi\in\mathrm{Diff}(\mathbb{R}^n)$ is specified so that the map
    \begin{align}\label{eq:brockett-sufficient-condition-general-correct}
        \psi:x\mapsto W{_{t_s}^{-\frac{1}{2}}}\varphi(W{_{t_s}^{\frac{1}{2}}} x)- x,
    \end{align}
    is a contraction, then there exists a (nonlinear) continuous feedback law $K_\cdot :\mathbb{R}^n\times[0,t_s]\rightarrow \mathbb{R}^m$ such that \eqref{eq:feedback-form-diff} holds for all $x_\init\in\mathbb{R}^n$ and all $t\in[0,t_s]$. 
\end{prop}
\begin{proof}
    We re-arrange the expression in \eqref{eq:optimal-state-diff} to obtain
    \begin{align*}
        x_\init&=\mathrm{e}^{-A_c t}x^\star_t-W_t W_{t_s}^{-\frac{1}{2}}\left(W_{t_s}^{-\frac{1}{2}}\varphi(x_\init)-W_{t_s}^{-\frac{1}{2}}x_\init\right),
    \end{align*}
    wherein we used \eqref{eq:pole-placement-periodictag}. Multiplying both sides by $W_{t_s}^{-\frac{1}{2}}$ and introducing the variable $y_\init=W_{t_s}^{-\frac{1}{2}}x_\init$, we obtain that
    \begin{align*}
        y_\init=W_{t_s}^{-\frac{1}{2}}&\mathrm{e}^{-A_c t}x^\star_t-W_{t_s}^{-\frac{1}{2}}W_t W_{t_s}^{-\frac{1}{2}}\psi(y_\init),
    \end{align*}
    for all $y_\init\in\mathbb{R}^n$ and all $t\in[0,t_s]$. It is now clear that the map
    \begin{align*}
        \eta: y_\init\mapsto W_{t_s}^{-\frac{1}{2}}&\mathrm{e}^{-A_c t}x^\star_t-W_{t_s}^{-\frac{1}{2}}W_t W_{t_s}^{-\frac{1}{2}}\psi(y_\init),
    \end{align*}
    is a contraction since
    \begin{align*}
        \|\eta(y_1)-\eta(y_2)\|&\leq \|W_{t_s}^{-\frac{1}{2}}W_t W_{t_s}^{-\frac{1}{2}}\|\|\psi(y_1)-\psi(y_2)\|\\
        &\leq \|\psi(y_1)-\psi(y_2)\|<\|y_1-y_2\|,
    \end{align*}
    wherein the last inequality holds due to the assumption that $\psi$ is a contraction. The remainder of the proof is identical to the argument in \cite[Lemma 1]{brockett2007optimal}. In particular, the claim of the proposition follows by invoking the Banach-fixed point theorem and the implicit function theorem.
\end{proof}

With this result a diffeomorphism that can be expressed as a composition of contractions with a sufficiently small Lipschitz constant can be reached by a trajectory of \eqref{eq:rightinvariant} through an appropriate choice of a control input.

\section{Concluding Remarks}

The controllability of dynamical systems has been a cornerstone concept of modern control theory. Thus, it may be surprising at first that technical issues 
remained unresolved, even for linear dynamics as in the formulation and in the class of problems addressed herein. The present work was influenced by Roger Brockett's work on the Liouville equation in \cite{brockett2007optimal}.
In fact, Brockett was amongst the pioneers who drew attention to the importance of quantifying control authority of systems that evolve on manifolds and Lie groups, and the present work falls within this general frame.

Whereas the controllability of the Liouville equation for linear dynamics is tightly linked to the controllability of dynamics evolving on $\GL$, studying directly dynamical flows on $\GL$ acquires a great deal of added practical significance as it models the evolution of a collection of $n$ identical dynamical systems. This problem of {\em collective steering} has been the main theme and motivation for this work.

A dual problem that is of equal importance is the estimation problem
to extract information about the flow of a collection of particles that obey the same dynamics, from knowledge of a collection of $n$ such particles--{\em tracer particles}. Tracer particles, typically, delineate a trajectory on $\GL$ that encapsulates information about the collective. Specifically, for a swarm of particles obeying identical linear dynamics, the state transition matrix $\Phi_t$ dictates the {\em conditional expectation}
\[
\mathbb E(X_t|X_0=x_0) = \Phi_t x_0,
\]
for particles in the collective. Thus, shaping a trajectory on $\GL$ or, estimating such a trajectory based on observations of the $n$ (or fewer) tracer particles at various points on the flight between terminal points in an interval, is of great importance. This topic, of conditioning the flow via conditional expectations will be taken up in forthcoming work, building on our recent framework in \cite{abdelgalil2024sub} and the results herein.

We conclude with a tribute to Roger Brockett, whose work sparked much of the development in the subject at hand. We have established controllability and strong controllability of the right invariant bilinear system \eqref{eq:right-invariant-system}, following up on his footsteps; to the best of knowledge, this is the first rigorous demonstration of these facts. Although the study of right invariant systems on Lie groups is a classical topic in geometric control, necessary and sufficient conditions for strong controllability in the presence of a non-trivial drift term are relatively sparse. Indeed, most results for systems with a non-trivial drift are primarily concerned with establishing the weaker notion of controllability and, often, rely on special properties of the underlying group such as compactness and semi-simplicity \cite[Chapter 6]{jurdjevic1997geometric}. Our interest in \eqref{eq:right-invariant-system} and, more generally, \eqref{eq:rightinvariant}, stems from their close connection to the theory of optimal mass transport and stochastic optimal mass transport (Schr\"odinger bridges).
We expect that the confluence of ideas from the parallel development of this theory, and the richness of the underlying geometry \cite{abdelgalil2024sub}, will provide fertile new ground for developments in control and estimation of dynamical systems.

\section*{Appendix: Auxiliary Lemmas}
\begin{lem}\label{lem:phi-decomposition}
    Assume $\Phi\in\GL$ and that  $\epsilon$ is a positive constant. Then, there exists a positive integer $N$ such that
\begin{align}\label{eq:phi-decomposition}
    \Phi&=\prod_{k=1}^N\Phi_k \mbox{ with } \|\Phi_k-I\|<\epsilon,
\end{align}
for all $k\in\{1,\ldots, N\}$.
\end{lem}
\begin{proof}
The claim can be argued based on the fact that $\GL$ is a connected Lie group, and thereby, there is a smooth path from the identity to any point. However, we chose to provide below an explicit construction based on a result in \cite{wustner2003connected}.
    Because $\GL$ is a real connected Lie group, it is shown in \cite{wustner2003connected} that any $\Phi\in\GL$ can be written as 
    \begin{align*}
        \Phi = \mathrm{e}^{M_2}\mathrm{e}^{M_1},
    \end{align*}
    for some matrices $M_1,M_2\in\mathbb{R}^{n\times n}$. Fix any such decomposition. From the series expansion of the exponential map, and for integers $N_1,N_2$, it can be seen
    \begin{align*}
        \|\mathrm{e}^{\frac{1}{N_\ell}M_\ell}-I\| \leq  \mathrm{e}^{\frac{1}{N_\ell}\|M_\ell\|} - 1,
    \end{align*}
    for $\ell\in\{1,2\}$. 
    We select integers
    \[
    N_\ell > \frac{\|M_\ell\|}{\log(1+\epsilon)},
    \]
    and $N=N_1+N_2$. Then the decomposition
    \begin{align*}
        \Phi &= \prod_{k=1}^{N}\Phi_k, & \Phi_k&=\begin{cases} \mathrm{e}^{\frac{1}{N_1}M_1}, & \text{ for } k\in\{1,\dots,N_1\},\\
        \mathrm{e}^{\frac{1}{N_2}M_2}, & \text{ for } k\in\{N_1+1,\dots,N\},
        \end{cases} 
    \end{align*}
    satisfies the conditions of the lemma.
\end{proof}

\begin{lem}\label{lem:lambdas}
    Assume that $\Delta\in\mathbb R^{n\times n}$ satisfies $\|\Delta\|<1$ and that $W$ is a positive definite symmetric matrix. Then,
    \[
    \|W^{\frac12}\Delta W^{-\frac12}\| <\sqrt{\frac{\lambda_{\rm max}(W)}{\lambda_{\rm min}(W)}}.
    \]
\end{lem}
\begin{proof}
    Clearly
    \[
    \frac{1}{\lambda_{\rm max}(W)}W\preceq I \preceq 
    \frac{1}{\lambda_{\rm min}(W)}W.
    \]
    Moreover, since $\Delta^\top \Delta \preceq I$,
    \[
    \frac{1}{\lambda_{\rm max}(W)}\Delta^\top W\Delta \preceq I,
    \]
    and therefore
     \[
    W^{-\frac12}\Delta^\top W\Delta W^{-\frac12} \preceq \frac{\lambda_{\rm max}(W)}{\lambda_{\rm min}(W)}I,
    \]
    from which the result follows.
\end{proof}

\begin{lem}\label{lem:lyapunov-equation-controllability}
    Assume that the bi-linear system \eqref{eq:right-invariant-system} is controllable on $\GL$. Then, the Lyapunov equation \eqref{eq:lyapunov-equation} is controllable on $\mathrm{Sym}^+(n)$.
\end{lem}
\begin{proof}
    Let $\Sigma_\init $ and $\Sigma_\fn$ be any two points in $\mathrm{Sym}^+(n)$. Since \eqref{eq:right-invariant-system} is controllable, there exists $t_\fn>0$ and $K_\cdot:[0,t_\fn]\rightarrow\mathbb{R}^{m\times n}$ such that the corresponding trajectory of \eqref{eq:right-invariant-system} satisfies the initial and final conditions
    \begin{align*}
        \Phi_{0}&=I, \mbox{ and } \Phi_{t_\fn}=\Sigma_\fn^{\frac{1}{2}}\Sigma_\init ^{-\frac{1}{2}}.
    \end{align*}
    Define the curve
    \begin{align*}
        \Sigma_{\cdot}: t\mapsto \Phi_t \Sigma_\init  \Phi_t^\top,
    \end{align*}
    Through direct differentiation, we see that $\Sigma_{\cdot}$ satisfies
    \begin{align*}
        \dot{\Sigma}_{t}&= \dot{\Phi}_t\Sigma_\init \Phi_t^\top + \Phi_t\Sigma_\init \dot{\Phi}_t^\top\\
        &=(A+B K_t)\Phi_t\Sigma_\init \Phi_t^\top + \Phi_t\Sigma_\init \Phi_t^\top(A+B K_t)^\top,\\
        &=(A+B K_t)\Sigma_{t} + \Sigma_{t}(A+B K_t)^\top
    \end{align*}
    for some $K_\cdot:[0,t_\fn]\rightarrow\mathbb{R}^{m\times n}$, which implies that $\Sigma_{\cdot}$ is a trajectory of the Lyapunov equation \eqref{eq:lyapunov-equation}, and, by construction, satisfies the initial and final conditions
    \begin{align*}
        \Sigma_0&=\Sigma_\init ,  \mbox{ and } \Sigma_{t_\fn}=\Sigma_\fn.
    \end{align*}
    Since $\Sigma_\init $ and $\Sigma_\fn$ are arbitrary points in $\mathrm{Sym}^+(n)$, it follows that \eqref{eq:lyapunov-equation} is controllable.   
\end{proof}


\bibliographystyle{plain}
\bibliography{References}

\end{document}